\documentclass[oneside,english]{amsart}
\usepackage[T1]{fontenc}
\usepackage{amsthm}
\usepackage{amstext}
\usepackage{setspace}
\usepackage{amssymb}
\usepackage{esint}

\makeatletter

\theoremstyle{plain}
\newtheorem{thm}{Theorem}[section]                                          
\newtheorem{prop}[thm]{Proposition}                          
\newtheorem{lem}[thm]{Lemma}
\newtheorem{cor}[thm]{Corollary}
\theoremstyle{definition}
\newtheorem{defn}[thm]{Definition}
\theoremstyle{remark}
\newtheorem*{rem*}{Remark}

\usepackage{mathrsfs}
\usepackage{fancyhdr}  
\usepackage{color}  
\usepackage{hyperref}  
\usepackage{graphicx}  
\usepackage{pstricks}
\usepackage{amssymb} 

\makeatother

\usepackage{babel}

\begin{document}

\title{Stein's Method and the Laplace Distribution}


\maketitle
\begin{center}
\author{John Pike\\ \small Department of Mathematics\\[-0.8ex] \small Cornell University\\ \small \texttt{jpike@cornell.edu}}
\par\end{center}

\begin{center}
\author{Haining Ren\\ \small Department of Mathematics\\[-0.8ex] \small University of Southern California\\ \small \texttt{hren@usc.edu}}
\par\end{center}

\begin{abstract}
Using Stein's method techniques, we develop a framework which allows one to bound the error terms 
arising from approximation by the Laplace distribution and apply it to the study of random sums of mean 
zero random variables. As a corollary, we deduce a Berry-Esseen type theorem for the convergence of 
certain geometric sums. Our results make use of a second order characterizing equation and a distributional 
transformation which is related to zero-biasing.
\end{abstract}

\maketitle

\section{Background and Introduction}

Beginning with the publication of Charles Stein's seminal paper \cite{Stn1}, probabilists and statisticians 
have developed a wide range of techniques based on characterizing equations for bounding the distance between 
the distribution of a random variable $X$ and and that of a random variable $Z$ having some specified target distribution.
The metrics for which these techniques are applicable are of the form 
$d_{\mathcal{H}}(\mathscr{L}(X),\mathscr{L}(Z))=\sup_{h\in\mathcal{H}}\left|\mathbb{E}[h(X)]-\mathbb{E}[h(Z)]\right|$
for some suitable class of functions $\mathcal{H}$, and include as special cases the Wasserstein, Kolmogorov, and 
total variation distances. (The Kolmogorov distance gives the $L^{\infty}$ distance between the associated distribution 
functions, so $\mathcal{H}=\{1_{(-\infty,a]}(x):a\in\mathbb{R}\}$. The total variation and Wasserstein distances 
correspond to letting $\mathcal{H}$ consist of indicators of Borel sets and $1$-Lipschitz functions, respectively.) The 
basic idea is to find an operator $\mathcal{A}$ such that $\mathbb{E}[(\mathcal{A}f)(X)]=0$ for all $f$ belonging to some
sufficiently large class of functions $\mathcal{F}$ if and only if $\mathscr{L}(X)=\mathscr{L}(Z)$. For example, Stein 
showed that $Z\sim\mathcal{N}(0,\sigma^{2})$ if and only if $\mathbb{E}[(\mathcal{A}_{N}f)(Z)]=\mathbb{E}[Zf(Z)-\sigma^{2}f'(Z)]=0$
for all absolutely continuous functions $f$ with $\left\Vert f'\right\Vert _{\infty}<\infty$ \cite{Stn1}, and shortly thereafter 
Louis Chen proved that $Z\sim\textrm{Poisson}(\lambda)$ if and only if 
$\mathbb{E}[(\mathcal{A}_{P}f)(Z)]=\mathbb{E}[Zf(Z)-\lambda f(Z+1)]=0$
for all functions $f$ for which the expectations exist \cite{Chen}. Similar characterizing operators have since been worked 
out for several other common distributions (e.g. \cite{Ehm,Luk,BrPh,PRR}). 

\pagebreak

Given a Stein operator $\mathcal{A}$ for $\mathscr{L}(Z)$, one then shows that for every $h\in\mathcal{H}$, 
the equation $(\mathcal{A}f)(x)=h(x)-\mathbb{E}[h(Z)]$ has solution $f_{h}\in\mathcal{F}$. 
Taking expectations, absolute values, and suprema yields
\[ d_{\mathcal{H}}(\mathscr{L}(X),\mathscr{L}(Z))=\sup_{h\in\mathcal{H}}\left|\mathbb{E}[h(X)]-\mathbb{E}[h(Z)]\right|
=\sup_{h\in\mathcal{H}}\left|\mathbb{E}[(\mathcal{A}f_{h})(X)]\right|.\]
(In practice, this is usually how one proves that $\mathbb{E}[(\mathcal{A}f)(X)]=0$ for $f\in\mathcal{F}$ is sufficient for $\mathscr{L}(X)=\mathscr{L}(Z)$.)

The intuition is that since $\mathbb{E}[(\mathcal{A}f)(Z)]=0$ for $f\in\mathcal{F}$, the distribution of $X$ should be close to 
that of $Z$ when $\mathbb{E}[(\mathcal{A}f)(X)]$ is close to zero. Remarkably, it is often easier to work with the right-hand side of 
the above equation, and the tools for analyzing distances between distributions in this manner are collectively known
as Stein's method. For more on this rich and fascinating subject, the authors recommend \cite{Stn2,Ross,CGS,DiaHol}.

In this paper, we apply the above ideas to the Laplace distribution. For $a\in\mathbb{R}$, $b>0$, a random 
variable $W\sim\textrm{Laplace}(a,b)$ has distribution function 
\[ F_{W}(w;a,b)=\begin{cases}
\frac{1}{2}e^{\frac{w-a}{b}}, & w\leq a\\
1-\frac{1}{2}e^{-\frac{w-a}{b}}, & w\geq a
\end{cases} \]
and density 
\[ f_{W}(w;a,b)=\frac{1}{2b}e^{-\frac{\left|w-a\right|}{b}},\textrm{ }w\in\mathbb{R}.\]
If $W\sim\textrm{Laplace}(0,b)$, then its moments are given by 
\[ \mathbb{E}[W^{k}]=\begin{cases}
0, & k\textrm{ is odd}\\
b^{k}k!, & k\textrm{ is even}
\end{cases}, \]
and its characteristic function is 
\[ \varphi_{W}(t)=\frac{1}{1+b^{2}t^{2}}.\]

This distribution was introduced by P.S. Laplace in 1774, four years prior to his proposal of the {}``second law of errors,'' 
now known as the normal distribution. Though nowhere near as ubiquitous as its younger sibling, the Laplace distribution 
appears in numerous applications, including image and speech compression, options pricing, and modeling sizes of sand 
particles, diamonds, and beans. For more properties and applications of the Laplace distribution, the reader is referred
to the text \cite{KKP}. 

Our interest in the Laplace distribution was sparked by the fact that if $X_{1},X_{2},...$ is a sequence of random variables 
(satisfying certain technical assumptions) and $N_{p}\sim\textrm{Geometric}(p)$ is independent of the $X_{i}$'s, then the 
sum $p^{\frac{1}{2}}\sum_{i=1}^{N_{p}}X_{i}$ converges weakly to the Laplace distribution as $p\rightarrow0$ \cite{KKP}.
Such geometric sums arise in a variety of settings \cite{Kal}, and the general setup (distributional convergence of sums of 
random variables) is exactly the type of problem for which one expects Stein's method computations to yield useful results. 
Indeed, Erol Pek\"{o}z and Adrian R\"{o}llin have applied Stein's method arguments to generalize a theorem due to R\'{e}nyi 
concerning the convergence of sums of a random number of positive random variables to the exponential distribution \cite{PeRo}.
By an analogous line of reasoning, we are able to carry out a similar program for convergence of random sums of certain 
mean zero random variables to the Laplace distribution.

\pagebreak

We begin in Section 2 by introducing a Stein operator which we show completely characterizes the mean zero Laplace 
distribution. Specifically, we prove

\begin{thm}
\label{1.1}
Let $W\sim\textnormal{Laplace}(0,b)$ and define the operator $\mathcal{A}$ by 
\[ (\mathcal{A}f)(x)=f(x)-f(0)-b^{2}f''(x). \]
Then $\mathbb{E}[(\mathcal{A}g)(W)]=0$ for every function $g$ such that $g$ and $g'$ are locally absolutely continuous and 
$\mathbb{E}\left|g'(W)\right|,\mathbb{E}\left|g''(W)\right|<\infty$. 
\smallskip \\
Conversely, if $X$ is any random variable such that $\mathbb{E}[(\mathcal{A}g)(X)]=0$ for every twice-differentiable function $g$ 
with $\left\Vert g\right\Vert _{\infty},\left\Vert g'\right\Vert _{\infty},\left\Vert g''\right\Vert _{\infty}<\infty$ then $X\sim\textnormal{Laplace}(0,b)$.
\end{thm}

In Section 3, we use this characterization to bound the distance to the Laplace distribution. For technical reasons, we work in the 
bounded Lipschitz metric, $d_{BL}$, which is defined in terms of $1$-Lipschitz test functions with sup norm $1$. We begin by defining 
the centered equilibrium transformation $X\mapsto X^{L}$ in terms of the functional equation 
$\mathbb{E}[f(X)]-f(0)=\frac{1}{2}\mathbb{E}[X^{2}]\mathbb{E}[f''(X^{L})]$ for all twice-differentiable functions $f$ such that 
$f$, $f'$, and $f''$ are bounded. After establishing that $X^{L}$ exists whenever $X$ has mean zero and finite nonzero variance, 
we derive the following coupling bound.

\begin{thm}
\label{1.2}
Suppose that $X$ is a random variable with $\mathbb{E}[X]=0$, $\mathbb{E}[X^{2}]=2b^{2}\in(0,\infty)$, and let $X^{L}$ have the centered equilibrium distribution for $X$. Then 
\[ d_{BL}(\mathscr{L}(X),\textnormal{Laplace}(0,b))\leq\frac{b+2}{b}\mathbb{E}\left|X-X^{L}\right|.\]
\end{thm}

Finally, in Section 4 we apply these tools to the study of random sums of mean zero random variables. As a special case, we show

\begin{thm}
\label{1.3}
Let $X_{1},X_{2},...$ be a sequence of independent random variables with 
$\mathbb{E}[X_{i}]=0$, $\mathbb{E}[X_{i}^{2}]=2b^{2}$, $\sup_{i\geq1}\mathbb{E}\left[\left|X_{i}\right|^{3}\right]=\rho<\infty$, 
and let $N\sim\textnormal{Geometric}(p)$ (with strictly positive support) be independent of the $X_{i}'s$. Then 
\[ d_{BL}\left(\mathscr{L}\left(p^{\frac{1}{2}}\sum_{i=1}^{N}X_{i}\right),\textnormal{Laplace}\left(0,b\right)\right)
\leq p^{\frac{1}{2}}\frac{b+2}{b}\left(b\sqrt{2}+\frac{\rho}{6b^{2}}\right)\]
for all\textup{ $p\in(0,1)$}.
\end{thm}

\section{Characterizing the Laplace Distribution}

Our first order of business is to establish a characterizing operator for the Laplace distribution. As is typical in Stein's method constructions, 
we split the proof of Theorem \ref{1.1} into two parts. We begin with 

\begin{lem}
\label{2.1}
Suppose that $W\sim\textnormal{Laplace}(0,b)$. If $g$ and $g'$ are locally absolutely continuous with 
$\mathbb{E}\left|g'(W)\right|,\mathbb{E}\left|g''(W)\right|<\infty$, then \[ \mathbb{E}[g(W)]-g(0)=b^{2}\mathbb{E}[g''(W)].\]
\end{lem}

\begin{proof}
Applying Fubini's theorem twice shows that 
\begin{align*}
\int_{0}^{\infty}g''(x)e^{-\frac{x}{b}}dx & =\int_{0}^{\infty}g''(x)\left(\int_{x}^{\infty}\frac{1}{b}e^{-\frac{y}{b}}dy\right)dx\\
 & =\frac{1}{b}\int_{0}^{\infty}\int_{0}^{y}g''(x)e^{-\frac{y}{b}}dxdy\\
 & =\frac{1}{b}\int_{0}^{\infty}g'(y)e^{-\frac{y}{b}}dy-\frac{g'(0)}{b}\int_{0}^{\infty}e^{-\frac{y}{b}}dy\\
 & =\frac{1}{b}\int_{0}^{\infty}g'(y)\left(\int_{y}^{\infty}\frac{1}{b}e^{-\frac{z}{b}}dz\right)dy-g'(0)\\
 & =\frac{1}{b^{2}}\int_{0}^{\infty}\int_{0}^{z}g'(y)e^{-\frac{z}{b}}dydz-g'(0)\\
 & =\frac{1}{b^{2}}\int_{0}^{\infty}g(z)e^{-\frac{z}{b}}dz-\frac{g(0)}{b^{2}}\int_{0}^{\infty}e^{-\frac{z}{b}}dz-g'(0)\\
 & =\frac{1}{b^{2}}\int_{0}^{\infty}g(z)e^{-\frac{z}{b}}dz-\frac{g(0)}{b}-g'(0).
\end{align*}

Setting $h(y)=g(-y)$, it follows from the previous calculation that
\begin{align*}
\int_{-\infty}^{0}g''(x)e^{\frac{x}{b}}dx & =\int_{0}^{\infty}g''(-y)e^{-\frac{y}{b}}dy=\int_{0}^{\infty}h''(y)e^{-\frac{y}{b}}dy\\
 & =\frac{1}{b^{2}}\int_{0}^{\infty}h(z)e^{-\frac{z}{b}}dz-\frac{h(0)}{b}-h'(0)\\
 & =\frac{1}{b^{2}}\int_{0}^{\infty}g(-z)e^{-\frac{z}{b}}dz-\frac{g(0)}{b}+g'(0)\\
 & =\frac{1}{b^{2}}\int_{-\infty}^{0}g(z)e^{\frac{z}{b}}dz-\frac{g(0)}{b}+g'(0).
\end{align*}

Summing the above expressions gives 
\begin{flalign*}
\mathbb{E}[g''(W)] & =\frac{1}{2b}\int_{-\infty}^{0}g''(x)e^{\frac{x}{b}}dx+\frac{1}{2b}\int_{0}^{\infty}g''(x)e^{-\frac{x}{b}}dx\\
 & =\frac{1}{2b}\left[\left(\frac{1}{b^{2}}\int_{-\infty}^{0}g(z)e^{\frac{z}{b}}dz-\frac{g(0)}{b}+g'(0)\right)\right.\\
 & \qquad\qquad\left.+\left(\frac{1}{b^{2}}\int_{0}^{\infty}g(z)e^{-\frac{z}{b}}dz-\frac{g(0)}{b}-g'(0)\right)\right]\\
 & =\frac{1}{2b}\left(\frac{1}{b^{2}}\int_{-\infty}^{\infty}g(z)e^{-\frac{\left|z\right|}{b}}dz-2\frac{g(0)}{b}\right)
=\frac{1}{b^{2}}\left(\mathbb{E}[g(W)]-g(0)\right).
\end{flalign*}
\end{proof}

Note that since the density of a $\textrm{Laplace}(0,b)$ random variable is given by 
$f_{W}(w)=\frac{1}{2b}e^{-\frac{\left|w\right|}{b}}$, the density method \cite{CGS} suggests the following characterizing 
equation for the Laplace distribution: 
\[ g'(w)-\frac{1}{b}\textrm{sgn}(w)g(w)=g'(w)+\frac{f_{W}'(w)}{f_{W}(w)}g(w)=0,\]
and indeed one can verify that if $W\sim\textrm{Laplace}(0,b)$,
then 
\[ \mathbb{E}[g'(W)]=\frac{1}{b}\mathbb{E}[\textrm{sgn}(W)g(W)]\]
for all absolutely continuous $g$ for which these expectations exist.
Thus if $g'$ is such a function as well, setting $G(w)=\textrm{sgn}(w)\left(g(w)-g(0)\right)$ gives
\begin{align*}
\mathbb{E}[g''(W)] & =\frac{1}{b}\mathbb{E}[\textrm{sgn}(W)g'(W)]=\frac{1}{b}\mathbb{E}[G'(W)]\\
 & =\frac{1}{b^{2}}\mathbb{E}[\textrm{sgn}(W)G(W)]=\frac{1}{b^{2}}\left(\mathbb{E}[g(W)]-g(0)\right),
\end{align*}
so the general form of the equation in Theorem \ref{1.1} can be ascertained by iterating the density method. 

Alternatively, it is known \cite{Ross} that if $Z\sim\textrm{Exponential}(1)$, then $\mathbb{E}[g'(Z)]=\mathbb{E}[g(Z)]-g(0)$ for all absolutely continuous $g$ 
with $\mathbb{E}\left|g'(Z)\right|<\infty$. Thus if $g'$ is also absolutely continuous and $\mathbb{E}\left|g''(Z)\right|<\infty$, then 
\[ \mathbb{E}[g''(Z)]=\mathbb{E}[g'(Z)]-g'(0)=\mathbb{E}[g(Z)]-g(0)-g'(0). \] 
Using this observation, one can derive the equation in Lemma \ref{2.1} by noting that if $J$ is independent of $Z$ with $\mathbb{P}(J=1)=\mathbb{P}(J=-1)=\frac{1}{2}$, then 
$W=bJZ$ has the $\textrm{Laplace}(0,b)$ distribution. 

We include each of these approaches because their analogues may be variously applicable in different situations involving the construction of characterizing equations. 
Observe that there is an iterative step involved in each case. By manipulating the integral defining $\mathbb{E}[g'(W)]$ or using the usual Stein equation for the exponential 
along with the representation $W=bJZ$, one arrives at the first order equation $\mathbb{E}[g'(W)]=\frac{1}{b}\mathbb{E}[\textrm{sgn}(W)g(W)]$ suggested by one application 
of the density method. However, we were not able to get much mileage out of the operator $(\widetilde{\mathcal{A}}g)(x)=g'(x)-\frac{1}{b}\textrm{sgn}(x)g(x)$, while the 
second-order operator $(\mathcal{A}g)(x)=g(x)-g(0)-b^{2}g''(x)$ turned out to be quite effective. This idea of iterating more traditional procedures to obtain higher order 
characterizing equations which are simpler to work with is one of the key insights of this paper.

Now, in order to establish the second part of Theorem \ref{1.1}, we will show that any $X$ satisfying the hypotheses has
\[ d_{BL}\left(\mathscr{L}(X),\textrm{Laplace}(0,b)\right)=0\]
where $d_{BL}$ denotes the bounded Lipschitz distance given by 
\[ d_{BL}\left(\mathscr{L}(X),\mathscr{L}(Y)\right)=\sup_{h\in\mathcal{H}_{BL}}\left|\mathbb{E}[h(X)]-\mathbb{E}[h(Y)]\right|,\]
\[ \mathcal{H}_{BL}=\{h:\left\Vert h\right\Vert _{\infty}\leq1\textrm{ and }\left|h(x)-h(y)\right|\leq\left|x-y\right|\textrm{ for all }x,y\in\mathbb{R}\}.\]
The claim will follow since $d_{BL}$ is a metric on the space of Borel probability measures on $\mathbb{R}$ \cite{vdvWel}.

In keeping with the general strategy laid out in the introduction, we consider the initial value problem 
\[ g(x)-b^{2}g''(x)=h(x)-Wh,\, g(0)=0\]
where $h\in\mathcal{H}_{BL}$ and $Wh:=\mathbb{E}[h(W)]$, $W\sim\textrm{Laplace}(0,b)$.

\begin{lem}
\label{2.2}
For $h\in\mathcal{H}_{BL}$, $W\sim\textnormal{Laplace}(0,b)$, $\widetilde{h}(x)=h(x)-Wh,$ a bounded, twice-differentiable 
solution to the initial value problem \textup{
\[ g(x)-b^{2}g''(x)=\widetilde{h}(x),\; g(0)=0\]
}is given by\textup{ 
\[ g_{h}(x)=\frac{1}{2b}\left(e^{\frac{x}{b}}\int_{x}^{\infty}e^{-\frac{y}{b}}\widetilde{h}(y)dy+e^{-\frac{x}{b}}\int_{-\infty}^{x}e^{\frac{y}{b}}\widetilde{h}(y)dy\right).\]
}This solution satisfies \textup{$\left\Vert g_{h}\right\Vert _{\infty}\leq2$,}
$\left\Vert g_{h}'\right\Vert _{\infty}\leq\frac{2}{b}$, $\left\Vert g_{h}''\right\Vert _{\infty}\leq\frac{4}{b^{2}}$
and $\left\Vert g_{h}'''\right\Vert _{\infty}\leq\frac{b+2}{b^{3}}.$ 
\end{lem}

\begin{proof}
The general solution to the homogeneous equation $g''(x)-b^{-2}g(x)=0$ is given by 
$g_{0}(x)=C_{1}e^{\frac{x}{b}}+C_{2}e^{-\frac{x}{b}}$, so, since the associated Wronskian is nonzero, 
the variation of parameters method suggests that a solution to the inhomogeneous equation 
$g''(x)-b^{-2}g(x)=-b^{-2}\widetilde{h}(x)$ is given by
\[ g_{h}(x)=u_{h}(x)e^{\frac{x}{b}}+v_{h}(x)e^{-\frac{x}{b}}\]
where 
\[ u_{h}(x)=\frac{1}{2b}\int_{x}^{\infty}e^{-\frac{y}{b}}\widetilde{h}(y)dy, \qquad 
v_{h}(x)=\frac{1}{2b}\int_{-\infty}^{x}e^{\frac{y}{b}}\widetilde{h}(y)dy.\]

Differentiation gives
\[ g_{h}'(x)=-\frac{1}{2b}\widetilde{h}(x)+\frac{1}{b}u_{h}(x)e^{\frac{x}{b}}+\frac{1}{2b}\widetilde{h}(x)-\frac{1}{b}v_{h}(x)e^{-\frac{x}{b}}
=\frac{1}{b}\left(u_{h}(x)e^{\frac{x}{b}}-v_{h}(x)e^{-\frac{x}{b}}\right), \]
and thus
\[ g_{h}''(x)=\frac{1}{b^{2}}\left(-\widetilde{h}(x)+u_{h}(x)e^{\frac{x}{b}}+v_{h}(x)e^{-\frac{x}{b}}\right)
=\frac{1}{b^{2}}\left(g_{h}(x)-\widetilde{h}(x)\right), \]
so $g_{h}$ is indeed a solution.

To see that the initial condition is satisfied, we observe that
\begin{align*}
g_{h}(0) & =u_{h}(0)+v_{h}(0)=\frac{1}{2b}\int_{-\infty}^{\infty}e^{-\frac{\left|y\right|}{b}}\widetilde{h}(y)dy=\int_{-\infty}^{\infty}f_{W}(y)\left(h(y)-Wh\right)dy\\
 & =\int_{-\infty}^{\infty}h(y)f_{W}(y)dy-Wh\int_{-\infty}^{\infty}f_{W}(y)dy=Wh-Wh=0.
\end{align*}

Moreover, since $\left\Vert h\right\Vert _{\infty}\leq1$, 
\[ \left|Wh\right|=\left|\frac{1}{2b}\int_{-\infty}^{\infty}h(x)e^{-\frac{\left|x\right|}{b}}dx\right|\leq\frac{1}{2b}\int_{-\infty}^{\infty}e^{-\frac{\left|x\right|}{b}}dx=1,\]
and thus $\left|\tilde{h}(x)\right|\leq\left|h(x)\right|+\left|Wh\right|\leq2$.
Consequently, 
\[ \left|u_{h}(x)e^{\frac{x}{b}}\right|\leq\frac{1}{2b}e^{\frac{x}{b}}\int_{x}^{\infty}2e^{-\frac{y}{b}}dy=1\]
and 
\[ \left|v_{h}(x)e^{-\frac{x}{b}}\right|\leq\frac{1}{2b}e^{-\frac{x}{b}}\int_{-\infty}^{x}2e^{\frac{y}{b}}dy=1,\]
for all $x\in\mathbb{R}$, and the bounds on $\left\Vert g_{h}\right\Vert _{\infty}$ and $\left\Vert g_{h}'\right\Vert _{\infty}$ follow.

Noting that $g_{h}''(x)=\frac{1}{b^{2}}\left(g_{h}(x)-\widetilde{h}(x)\right)$ and thus 
$g_{h}'''(x)=\frac{1}{b^{3}}g_{h}'(x)-\frac{1}{b^2}h'(x)$ completes the proof 
since $\left\Vert \tilde{h}\right\Vert _{\infty}\leq 2$ and $\left\Vert h'\right\Vert _{\infty}\leq 1$.
\end{proof}

\pagebreak

With the preceding result in hand, we can finish of the proof of Theorem \ref{1.1} via

\begin{lem}
\label{2.3}
If $X$ is a random variable such 
\[ \mathbb{E}[g(X)]-g(0)=b^{2}\mathbb{E}[g''(X)]\]
for every twice-differentiable function $g$ with $\left\Vert g\right\Vert _{\infty},\left\Vert g'\right\Vert _{\infty},\left\Vert g''\right\Vert _{\infty}<\infty$,
then $X\sim\textnormal{Laplace}(0,b)$.
\end{lem}

\begin{proof}
Let $W\sim\textrm{Laplace}(0,b)$ and, for $h\in\mathcal{H}_{BL}$, let $g_{h}$ be as in Lemma \ref{2.2}. Because $g_{h}(0)=0$
and $g_{h},g_{h}',g_{h}''$ are bounded, it follows from the above assumptions that 
\[ \mathbb{E}[h(X)]-\mathbb{E}[h(W)]=\mathbb{E}[g_{h}(X)-b^{2}g_{h}''(X)]=0.\]
Taking the supremum over $h\in\mathcal{H}_{BL}$ shows that $d_{BL}\left(\mathscr{L}(X),\mathscr{L}(W)\right)=0$.
\end{proof}

Before moving on, we observe that the reason we are working with the bounded Lipschitz distance is that the bounds on $g_{h}$ and its derivatives depended 
on both $h$ and $h'$ having finite sup norm. As $d_{BL}$ is not especially common (at least explicitly) in the Stein's method literature, we conclude this section 
with a proposition relating it to the more familiar Kolmogorov distance 
\[ d_{K}\left(\mathscr{L}(X),\mathscr{L}(Y)\right)=\sup_{x\in\mathbb{R}}\left|\mathbb{P}\{X\leq x\}-\mathbb{P}\{Y\leq x\}\right|.\]

\begin{prop}
\label{2.4}
If $Z$ is an absolutely continuous random variable whose density, $f_{Z}$, is uniformly bounded by a constant $C<\infty$, 
then for any random variable $X$,
\[ d_{K}\left(\mathscr{L}(X),\mathscr{L}(Z)\right)\leq\frac{C+2}{2}\sqrt{d_{BL}\left(\mathscr{L}(X),\mathscr{L}(Z)\right)}.\]
\end{prop}

\begin{proof}
We first note that the inequality holds trivially if $d_{BL}\left(\mathscr{L}(X),\mathscr{L}(Z)\right)=0$ as $d_{BL}$ and 
$d_{K}$ are metrics. Also, since $d_{K}(P,Q)\leq1$ for all probability measures $P$ and $Q$, $\frac{C+2}{2}\geq1$,
and $d_{BL}\left(\mathscr{L}(X),\mathscr{L}(Z)\right)\geq1$ implies $\sqrt{d_{BL}\left(\mathscr{L}(X),\mathscr{L}(Z)\right)}\geq1$, 
we have 
\begin{alignat*}{1}
d_{K}\left(\mathscr{L}(X),\mathscr{L}(Z)\right) & \leq1\leq\frac{C+2}{2}\sqrt{d_{BL}\left(\mathscr{L}(X),\mathscr{L}(Z)\right)}
\end{alignat*}
whenever $d_{BL}\left(\mathscr{L}(X),\mathscr{L}(Z)\right)\geq1$. Thus it suffices to consider the case where 
$d_{BL}\left(\mathscr{L}(X),\mathscr{L}(Z)\right)\in(0,1)$.

Now, for $x\in\mathbb{R}$, $\varepsilon>0$, write 
\[ h_{x}(z)=1_{(-\infty,x]}(z)=\begin{cases}
1, & z\leq x\\
0, & z>x\end{cases} 
\textrm{ and  }h_{x,\varepsilon}(z)=\begin{cases}
1, & z\leq x\\
1-\frac{z-x}{\varepsilon}, & z\in(x,x+\varepsilon]\\
0, & z>x+\varepsilon
\end{cases}. \]
Then for all $x\in\mathbb{R}$, 
\begin{align*}
\mathbb{E}[h_{x}(X)-h_{x}(Z)] & =\mathbb{E}[h_{x}(X)]-\mathbb{E}[h_{x,\varepsilon}(Z)]+\mathbb{E}[h_{x,\varepsilon}(Z)]-\mathbb{E}[h_{x}(Z)]\\
 & \leq\left(\mathbb{E}[h_{x,\varepsilon}(X)]-\mathbb{E}[h_{x,\varepsilon}(Z)]\right)+\int_{x}^{x+\varepsilon}\left(1-\frac{z-x}{\varepsilon}\right)f_{Z}(z)dz\\
 & \leq\left|\mathbb{E}[h_{x,\varepsilon}(X)]-\mathbb{E}[h_{x,\varepsilon}(Z)]\right|+\frac{C\varepsilon}{2}.
\end{align*}
Since $d_{BL}\left(\mathscr{L}(X),\mathscr{L}(Z)\right)\in(0,1)$, if we take $\varepsilon=\sqrt{d_{BL}\left(\mathscr{L}(X),\mathscr{L}(Z)\right)}\in(0,1)$,
then $\varepsilon h_{x,\varepsilon}\in\mathcal{H}_{BL}$ and thus
\begin{align*}
\mathbb{E}[h_{x}(X)-h_{x}(Z)] & \leq\frac{1}{\varepsilon}\left|\mathbb{E}[\varepsilon h_{x,\varepsilon}(X)]
-\mathbb{E}[\varepsilon h_{x,\varepsilon}(Z)]\right|+\frac{C\varepsilon}{2}\\
 & \leq\frac{1}{\varepsilon}d_{BL}\left(\mathscr{L}(X),\mathscr{L}(Z)\right)+\frac{C\varepsilon}{2}\\
 & =\frac{C+2}{2}\sqrt{d_{BL}\left(\mathscr{L}(X),\mathscr{L}(Z)\right)}.
\end{align*}
A similar argument using 
\begin{align*}
\mathbb{E}[h_{x}(Z)-h_{x}(X)] & =\mathbb{E}[h_{x}(Z)]-\mathbb{E}[h_{x-\varepsilon,\varepsilon}(Z)]
+\mathbb{E}[h_{x-\varepsilon,\varepsilon}(Z)]-\mathbb{E}[h_{x}(X)]\\
 & \leq\frac{C\varepsilon}{2}+\left(\mathbb{E}[h_{x-\varepsilon,\varepsilon}(Z)]-\mathbb{E}[h_{x-\varepsilon,\varepsilon}(X)]\right)
\end{align*}
shows that 
\[ \left|\mathbb{E}[h_{x}(X)]-\mathbb{E}[h_{x}(Z)]\right|\leq\frac{C+2}{2}\sqrt{d_{BL}\left(\mathscr{L}(X),\mathscr{L}(Z)\right)}\]
for all $x\in\mathbb{R}$, and the proposition follows by taking suprema.
\end{proof}

\begin{rem*}
When $C\geq1$, we can take $\varepsilon=\sqrt{\frac{1}{C}d_{BL}\left(\mathscr{L}(X),\mathscr{L}(Z)\right)}$
in the above argument to obtain an improved bound of 
\[ d_{K}\left(\mathscr{L}(X),\mathscr{L}(Z)\right)\leq\frac{3}{2}\sqrt{Cd_{BL}\left(\mathscr{L}(X),\mathscr{L}(Z)\right)}.\]
\end{rem*}

To the best of the authors' knowledge, the above proposition is original, though the proof follows the same basic line of 
reasoning as the well-known bound on the Kolmogorov distance by the Wasserstein distance (see Proposition 1.2 in 
\cite{Ross}). It seems that the primary reason for using the Wasserstein metric, $d_{W}$, is that it enables one to work 
with smoother test functions while still implying convergence in the more natural Kolmogorov distance. Proposition \ref{2.4} shows 
that $d_{BL}$ also upper-bounds $d_{K}$ while enjoying all of the resulting smoothness of Wasserstein test functions and with additional 
boundedness properties to boot. Moreover, the Wasserstein distance is not always well-defined (e.g. if one of the distributions 
does not have a first moment), whereas $d_{BL}$ always exists. Finally, $d_{BL}$ is a fairly natural measure of distance since it 
metrizes weak convergence \cite{vdvWel}. However, we always have 
$d_{W}\left(\mathscr{L}(X),\mathscr{L}(Z)\right)\geq d_{BL}\left(\mathscr{L}(X),\mathscr{L}(Z)\right)$,
and it is possible for a sequence to converge in $d_{BL}$ but not in $d_{W}$ or $d_{K}$. Furthermore, the bounded 
Lipschitz metric does not scale as nicely as the Kolmogorov or Wasserstein distances when the associated random variables 
are multiplied by a positive constant. For the remainder of this paper, we will state our results in terms of $d_{BL}$ with the 
corresponding Kolmogorov bound being implicit therein, though one should note that, as with the Wasserstein bound on 
$d_{K}$, Kolmogorov bounds obtained in this fashion are not necessarily optimal, often giving the root of the true rate.

\section{The Centered Equilibrium Transformation}

Our next task is to use the characterization in Theorem \ref{1.1} to obtain bounds on the error terms resulting from approximation by the Laplace distribution. 
To this end, we introduce the following definition.

\begin{defn}
\label{3.1}
For any nondegenerate random variable $X$ with mean zero and finite variance, we say that the random variable $X^{L}$ has the 
\textit{centered equilibrium distribution} with respect to $X$ if 
\[ \mathbb{E}[f(X)]-f(0)=\frac{1}{2}\mathbb{E}[X^{2}]\mathbb{E}[f''(X^{L})]\]
for all twice-differentiable functions $f$ such that $f$, $f'$, and $f''$ are bounded. We call the map $X\mapsto X^{L}$ the \textit{centered equilibrium transformation}. 
\end{defn}

Note that the centered equilibrium distribution is uniquely defined because \linebreak 
$\mathbb{E}[f''(X)]=\mathbb{E}[f''(Y)]$ for all twice continuously differentiable functions with compact support 
implies $X=_{d}Y$. The nomenclature is in reference to the equilibrium distribution from renewal theory which was used in a similar manner in \cite{PeRo} for a related problem involving 
the exponential distribution. 

Since the characterizing equation for the Laplace distribution involves second derivatives and a variance term, one expects some kind of relation 
between $X^{L}$ and the zero bias distribution for $X$ (defined by $\mathbb{E}[Xf(X)]=\mathbb{E}[X^{2}]\mathbb{E}[f'(X^{z})]$ for all absolutely 
continuous $f$ for which $\mathbb{E}\left|Xf(X)\right|<\infty$ \cite{GoRe1}), in much the same way as the equilibrium distribution is related to the size bias distribution \cite{PeRo}. 
The following theorem shows that this is indeed the case. Moreover, since the zero bias distribution is defined for any $X$ with $\mathbb{E}[X]=0$ and 
$\textrm{Var}(X)\in(0,\infty)$, it will follow that every such random variable has a centered equilibrium distribution.

\begin{thm}
\label{3.2}
Suppose that $X$ has mean zero and variance $\sigma^{2}\in(0,\infty)$. Let $X^{z}$ have the zero bias distribution with respect to $X$ and let 
$B\sim\textnormal{Beta}(2,1)$ be independent of $X^{z}$. Then $X^{L}:=BX^{z}$ satisfies 
\[ \mathbb{E}[f(X)]-f(0)=\frac{\sigma^{2}}{2}\mathbb{E}[f''(X^{L})]\]
for all twice-differentiable $f$ with $\left\Vert f\right\Vert _{\infty},\left\Vert f'\right\Vert _{\infty}, \left\Vert f''\right\Vert _{\infty}<\infty$.
\end{thm}

\begin{proof}
Applying the fundamental theorem of calculus, Fubini's theorem, the definition of $X^{z}$, and the fact that $B$ has density $p(x)=2x1_{[0,1]}(x)$ gives 
\begin{align*}
\mathbb{E}[f(X)]-f(0) & =\mathbb{E}\left[\int_{0}^{1}Xf'(uX)du\right]=\int_{0}^{1}\mathbb{E}\left[Xf'(uX)\right]du\\
 & =\sigma^{2}\int_{0}^{1}u\mathbb{E}[f''(uX^{z})]du=\sigma^{2}\mathbb{E}\left[\int_{0}^{1}uf''(uX^{z})du\right]\\
 & =\frac{\sigma^{2}}{2}\mathbb{E}\left[f''(BX^{z}\right]
\end{align*}
The assumptions ensure that all of the functions are integrable.
\end{proof}

\begin{cor}
\label{3.3}
If $X$ has variance $\mathbb{E}[X^2]=2b^{2}$ and $X^{L}$ has the centered equilibrium distribution with respect to $X$, then $X^{L}$ is absolutely continuous with density 
\[ f_{X^{L}}(x)=\frac{1}{b^{2}}\int_{0}^{1}\mathbb{E}\left[X;X>\frac{x}{v}\right]dv. \]
\end{cor}

\begin{proof}
$X^{L}=_{d}BX^{z}$ with $B$ and $X^{z}$ as in Theorem \ref{3.2}, so the claim follows from the fact \cite{GoRe1} that $X^{z}$ is absolutely continuous with density 
$f_{X^{z}}(x)=\frac{1}{2b^{2}}\mathbb{E}[X;X>x]$ by the usual method of computing the density of a product.
\end{proof}

\begin{rem*}
In an earlier version of this paper, we established the existence of the centered equilibrium distribution by showing that for certain random variables $X$, $X^{L}$ can be obtained by 
iterating the $X-P$ bias transformation from \cite{GoRe2} with $P(x)=\textrm{sgn}(x)$. Though there may be some merit to such a strategy and it provides another example of how 
results for higher order Stein operators may be obtained by iterating more traditional techniques, in our case it required the rather artificial assumption that the variates in the domain of 
the transformation have median zero. Those interested in the iterated $X-P$ bias approach are referred to the article \cite{Dob} by Christian D\"{o}bler, which contains the essential 
technical details of our original argument. 
\end{rem*}

Lemma \ref{2.3} shows that, up to scaling, the mean zero Laplace distribution is the unique fixed point of the centered equilibrium transformation. 
Thus one expects that if a random variable is close to its centered equilibrium transform, then its distribution is close to the Laplace. 
Theorem \ref{1.2} formalizes this intuition.

\begin{proof}[Proof of Theorem 1.2]
If $X$ is a random variable with $\mathbb{E}[X^{2}]=2b^{2}<\infty$ and $X^{L}$ has the centered equilibrium distribution for $X$, then for all 
$h\in\mathcal{H}_{BL},$ taking $g_{h}$ as in Lemma \ref{2.2}, we see that 
\begin{align*}
\left|Wh-\mathbb{E}[h(X)]\right| & =\left|\mathbb{E}[g_{h}(X)-b^{2}g_{h}''(X)]\right|=\left|\mathbb{E}[b^{2}g_{h}''(X^{L})-b^{2}g_{h}''(X)]\right|\\
 & \leq b^{2}\mathbb{E}\left|g_{h}''(X^{L})-g_{h}''(X)\right|\leq b^{2}\left\Vert g_{h}'''\right\Vert \mathbb{E}\left|X^{L}-X\right|\\
 & =\frac{b+2}{b}\mathbb{E}\left|X-X^{L}\right|.
\end{align*}
\end{proof}

For the example in Section 4, we will also need the following complementary result.

\begin{prop}
\label{3.4}
If $Y^{L}$ has the centered equilibrium distribution for $Y$ and $\mathbb{E}[Y^{2}]=2b^{2}$, then
\[ \mathbb{E}\left|Y-Y^{L}\right|\leq \mathbb{E}\left|Y\right|+\frac{1}{6b^{2}}\mathbb{E}[|Y|^{3}].\]
\end{prop}

\begin{proof}
We may assume that $\mathbb{E}[|Y|^{3}]<\infty$ as the inequality is trivial otherwise. The result will follow immediately from the triangle 
inequality if we can show that $\mathbb{E}\left|Y^{L}\right|=\frac{1}{6b^{2}}\mathbb{E}\left|Y^{3}\right|$, which is
what one would obtain by formally plugging $f(y)=\left|y\right|^{3}$ into the definition of the transformation $Y\mapsto Y^{L}$.

Of course, neither $f$, $f'$, nor $f''$ is bounded, so we must proceed by approximation. To this end, define
\[ f_{n}(x)=\begin{cases}
\left|x\right|^{3},&  \left|x\right|\leq n\\
n^{3}+3n^{2}(\left|x\right|-n)-\frac{3}{2}(\left|x\right|-n)^{2},&  n<\left|x\right|\leq n^{2}+n\\
n^{3}+3n^{2}(2n^{2}+n-\left|x\right|)-\frac{3}{2}(2n^{2}+n-\left|x\right|)^{2},&  n^{2}+n<\left|x\right|\leq2n^{2}+n\\
(2n^{2}+2n-\left|x\right|)^{3},&  2n^{2}+n<\left|x\right|\leq2n^{2}+2n\\
0,&  \left|x\right|>2n^{2}+2n
\end{cases} \]
By construction, $f_{n}$ is smooth with compact support and satisfies 
$\left|f_{n}^{(i)}\right|\leq\left|f^{(i)}\right|$ for all $n\in\mathbb{N}$, $i=0,1,2$, thus fulfilling the conditions in the 
definition of the centered equilibrium transformation when $\mathbb{E}[|Y|^{3}]<\infty$. Moreover, $f_{n}^{(i)}\rightarrow f^{(i)}$ 
pointwise for $i=0,1,2$, so it follows from the dominated convergence theorem that 
\[ \mathbb{E}[f(Y)]=\lim_{n\rightarrow\infty}\mathbb{E}[f_{n}(Y)]=\lim_{n\rightarrow\infty}b^{2}\mathbb{E}[f_{n}''(Y^{L})].\]
Fatou's lemma shows that $f''(Y^{L})$ is integrable since 
\[ \mathbb{E}\left[f''(Y^{L})\right]=\mathbb{E}\left[\liminf_{n\rightarrow\infty}f_{n}''(Y^{L})\right]\leq\liminf_{n\rightarrow\infty}\mathbb{E}\left[f_{n}''(Y^{L})\right]
=\frac{1}{b^{2}}\mathbb{E}[f(Y)]<\infty,\]
so another application of dominated convergence gives 
\[ \mathbb{E}[f(Y)]=\lim_{n\rightarrow\infty}b^{2}\mathbb{E}[f_{n}''(Y^{L})]=b^{2}\mathbb{E}[f''(Y^{L})].\]
\end{proof}

The same general argument shows that if $Y^{L}$ has the centered equilibrium distribution for $Y$ and $\mathbb{E}\left[\left|Y\right|^{n}\right]<\infty$, 
then $\mathbb{E}[q(Y)]-q(0)=\frac{1}{2}\mathbb{E}[Y^{2}]\mathbb{E}[q''(Y^{L})]$ for any polynomial $q$ of degree at most $n$. 
As is often the case with distributional transformations defined in terms of functional equations, we find it more convenient to define $X\mapsto X^{L}$ 
in terms of a relatively small class of test functions and then argue by approximation when we want to apply the relation more generally.

\section{Random Sums}

The {$p$}-geometric summation of a sequence of random variables $X_{1},X_{2},...$ is defined as 
$S_{p}=X_{1}+X_{2}+...+X_{N_{p}}$ where $N_{p}$ is geometric with success probability $p$ - that is,
$\mathbb{P}\{N_{p}=n\}=p(1-p)^{n-1}$, $n\in\mathbb{N}$ - and is independent of all else. A result due to 
R\'{e}nyi \cite{Ren} states that if $X_{1},X_{2},...$ are i.i.d., positive, nondegenerate random variables with 
$\mathbb{E}[X_{i}]=1$, then $ $$\mathscr{L}(pS_{p})\rightarrow\textrm{Exponential}(1)$ as $p\rightarrow0$. In fact, just as 
the normal law is the only nondegenerate distribution with finite variance that is stable under ordinary summation
(in the sense that if $X,X_{1},X_{2},...$ are i.i.d. nondegenerate random variables with finite variance, then for every 
$n\in\mathbb{N}$, there exist $a_{n}>0$, $b_{n}\in\mathbb{R}$ such that $X=_{d}a_{n}\left(X_{1}+...+X_{n}\right)+b_{n}$),
it can be shown that if $X,X_{1},X_{2},...$ are i.i.d., positive, and nondegenerate with finite variance, then there exists 
$a_{p}>0$ such that $a_{p}\left(X_{1}+...+X_{N_{p}}\right)=_{d}X$ for all $p\in(0,1)$
if and only if $X$ has an exponential distribution. Similarly, if we assume that $Y,Y_{1},Y_{2},...$ are i.i.d., symmetric, 
and nondegenerate with finite variance, then there exists $a_{p}>0$ such that $a_{p}\left(Y_{1}+...+Y_{N_{p}}\right)=_{d}Y$
for all $p\in(0,1)$ if and only if $Y$ has a Laplace distribution. Moreover, it must be the case that $a_{p}=p^{\frac{1}{2}}$. 
In addition, we have an analog of R\'{e}nyi's theorem \cite{KKP}:

\begin{thm}
\label{4.1}
Suppose that $X_{1},X_{2},...$ are i.i.d., symmetric, and nondegenerate random variables with finite variance 
$\sigma^{2}$, and let $N_{p}\sim\textrm{Geometric}(p)$ be independent of the $X_{i}'s$. If 
\[ a_{p}\sum_{i=1}^{N_{p}}X_{i}\rightarrow_{d}X\;\textrm{as }p\rightarrow0,\]
then there exists $\gamma>0$ such that $a_{p}=p^{\frac{1}{2}}\gamma+o(p^{\frac{1}{2}})$ and $X$ has the Laplace 
distribution with mean $0$ and variance $\sigma^{2}\gamma^{2}$.
\end{thm}

\pagebreak 

A recent theorem due to Alexis Toda \cite{Tod} gives the following Lindeberg-type conditions for the existence of the 
distributional limit in Theorem \ref{4.1}.

\begin{thm}[Toda]
\label{4.2}
 Let $X_{1},X_{2},...$ be a sequence of independent (but not necessarily identically distributed) random variables such that 
$\mathbb{E}[X_{i}]=0$ and $\textnormal{Var}(X_{i})=\sigma_{i}^{2}$, and let $N_{p}\sim\textnormal{Geometric}(p)$
independent of the $X_{i}'s$. Suppose that 
\[ \lim_{n\rightarrow\infty}n^{-\alpha}\sigma_{n}^{2}=0\:\mathit{for\, some\,0<\alpha<1,}\]
\[ \sigma^{2}:=\lim_{n\rightarrow\infty}\frac{1}{n}\sum_{i=1}^{n}\sigma_{i}^{2}>0\mathit{\: exists,}\]
and for all $\epsilon>0$, 
\[ \lim_{p\rightarrow0}\sum_{i=1}^{\infty}(1-p)^{i-1}p\mathbb{E}[X_{i}^{2};\left|X_{i}\right|\geq\epsilon p^{-\frac{1}{2}}]=0.\]
Then as $p\rightarrow0$, the sum $p^{\frac{1}{2}}\sum_{i=1}^{N_{p}}X_{i}$ converges weakly to the Laplace 
distribution with mean $0$ and variance $\sigma^{2}$.
\end{thm}

\begin{rem*}
The original statement of Toda's theorem is slightly more general, allowing for convergence to a possibly asymmetric 
Laplace distribution.
\end{rem*}

In 2011, Pek\"{o}z and R\"{o}llin were able to generalize R\'{e}nyi's theorem by using a distributional transformation inspired 
by Stein's method considerations \cite{PeRo}. Specifically, for a nonnegative random variable $X$ with $\mathbb{E}[X]<\infty$, they 
say that $X^{e}$ has the equilibrium distribution with respect to $X$ if $\mathbb{E}[f(X)]-f(0)=\mathbb{E}[X]\mathbb{E}[f'(X^{e})]$ for all Lipschitz $f$ 
and use this to bound the Wasserstein and Kolmogorov distances to the $\textrm{Exponential}(1)$ distribution. 
The equilibrium distribution arises in renewal theory, but its utility in analyzing convergence to the exponential 
distribution comes from the fact that a Stein operator for the exponential distribution with mean one is given by 
$(\mathcal{A}_{E}f)(x)=f'(x)-f(x)+f(0)$, so $X\sim\textrm{Exponential}(1)$ is the unique fixed point of the equilibrium
transformation \cite{Ross}. This transformation and the similarity between our characterization of the Laplace and 
the above characterization of the exponential inspired our construction of the centered equilibrium transformation, and the 
fact that both distributions are stable under geometric summation led us to parallel their argument for bounding the distance 
between $p$-geometric sums of positive random variables and the exponential distribution in order to obtain corresponding
results for the Laplace case. Our results are summarized in the following theorem.

\begin{thm}
\label{4.3}
Let $N$ be any $\mathbb{N}$-valued random variable with $\mu=\mathbb{E}[N]<\infty$ and let $X_{1},X_{2},...$ be a sequence of 
independent random variables, independent of $N$, with $\mathbb{E}[X_{i}]=0$, and $\mathbb{E}[X_{i}^{2}]=\sigma_{i}^{2}\in(0,\infty)$. Set
$\sigma^{2}=\mathbb{E}\left[\left(\sum_{i=1}^{N}X_{i}\right)^{2}\right]=\mathbb{E}\left[\sum_{i=1}^{N}\sigma_{i}^{2}\right]$
and let $M$ be any $\mathbb{N}$-valued random variable, independent of the $X_{i}'s$ and defined on the same space as 
$N$, satisfying 
\[ \mathbb{P}\{M=m\}=\frac{\sigma_{m}^{2}}{\sigma^{2}}\mathbb{P}\{N\geq m\}.\]
Then 
\begin{multline*}
d_{BL}\left(\mathscr{L}\left(\mu^{-\frac{1}{2}}\sum_{i=1}^{N}X_{i}\right),\textnormal{Laplace}\left(0,\frac{\sigma}{\sqrt{2\mu}}\right)\right)\\
\leq\left(\mu^{-\frac{1}{2}}+\frac{\sqrt{8}}{\sigma}\right)\left(\mathbb{E}\left|X_{M}-X_{M}^{L}\right|
+\sup_{i\geq1}\sigma_{i}\mathbb{E}\left[\left|N-M\right|^{\frac{1}{2}}\right]\right).
\end{multline*}
\end{thm}

\begin{proof}
We first note that 
\[ \sigma^{2}=\sum_{m=1}^{\infty}\mathbb{P}\{N=m\}\sum_{i=1}^{m}\sigma_{i}^{2}=\sum_{m=1}^{\infty}\mathbb{P}\{N\geq m\}\sigma_{m}^{2},\]
so $M$ is well-defined. 

Now, taking $V=\mu^{-\frac{1}{2}}\sum_{i=1}^{N}X_{i}$, we claim that 
$V^{L}=\mu^{-\frac{1}{2}}\left(\sum_{i=1}^{M-1}X_{i}+X_{M}^{L}\right)$
has the centered equilibrium distribution with respect to $V$. (Throughout, $X_{m}^{L}$ is taken to be independent of $M$, $N$, and $X_{k}$ for $k\neq m$.) 
To see that this is so, let $f$ be any function satisfying the assumptions in Definition \ref{3.1}. Then, using the notation 
\[ \mathbb{X}=\{X_{i}\}_{i\geq1},\; g_{m}(\mathbb{X})=f\left(\mu^{-\frac{1}{2}}\sum_{i=1}^{m}X_{i}\right),\; f_{s}(x)=f\left(\mu^{-\frac{1}{2}}s+\mu^{-\frac{1}{2}}x\right),\]
letting $\nu_{m}$ denote the distribution of 
\[ S_{m-1}=\sum_{i=1}^{m-1}X_{i},\]
and observing that, by independence, 
\begin{alignat*}{1}
\mathbb{E}[h''(X_{m}^{L})|S_{m-1}=s] & =\mathbb{E}[h''(X_{m}^{L})]=\frac{2}{\sigma_{m}^{2}}\left(\mathbb{E}[h(X_{m})-h(0)]\right)\\
 & =\frac{2}{\sigma_{m}^{2}}\mathbb{E}[h(X_{m})-h(0)|S_{m-1}=s]
\end{alignat*}
for all $s\in\mathbb{R}$ and all twice differentiable $h$ with $h,h',h''$ bounded, we see that 
\begin{align*}
\mathbb{E}\left[f''\left(\mu^{-\frac{1}{2}}\sum_{i=1}^{m-1}X_{i}+\mu^{-\frac{1}{2}}X_{m}^{L}\right)\right] 
& =\int \mathbb{E}\left[f''(\mu^{-\frac{1}{2}}s+\mu^{-\frac{1}{2}}X_{m}^{L})\left|S_{m-1}=s\right.\right]d\nu_{m}(s)\\
 & =\int \mathbb{E}\left[\mu f_{s}''(X_{m}^{L})\left|S_{m-1}=s\right.\right]d\nu_{m}(s)\\
 & =\frac{2\mu}{\sigma_{m}^{2}}\int \mathbb{E}\left[f_{s}(X_{m})-f_{s}(0)\left|S_{m-1}=s\right.\right]d\nu_{m}(s)\\
 & =\frac{2\mu}{\sigma_{m}^{2}}\int \mathbb{E}\left[f(\mu^{-\frac{1}{2}}s+\mu^{-\frac{1}{2}}X_{m})\right.\\
 & \qquad\qquad\qquad\left.-f(\mu^{-\frac{1}{2}}s)\left|S_{m-1}=s\right.\right]d\nu_{m}(s)\\
 & =\frac{2\mu}{\sigma_{m}^{2}}\mathbb{E}\left[f\left(\mu^{-\frac{1}{2}}\sum_{i=1}^{m}X_{i}\right)-f\left(\mu^{-\frac{1}{2}}\sum_{i=1}^{m-1}X_{i}\right)\right]\\
 & =\frac{2\mu}{\sigma_{m}^{2}}\mathbb{E}\left[g_{m}(\mathbb{X})-g_{m-1}(\mathbb{X})\right]
\end{align*}
for all $m\in\mathbb{N}$, hence 
\begin{align*}
\mathbb{E}[f''(V^{L})] & =\sum_{m=1}^{\infty}\mathbb{P}\{M=m\}\mathbb{E}\left[f''\left(\mu^{-\frac{1}{2}}\sum_{i=1}^{m-1}X_{i}+\mu^{-\frac{1}{2}}X_{m}^{L}\right)\right]\\
 & =\frac{2\mu}{\sigma^{2}}\sum_{m=1}^{\infty}\frac{\sigma^{2}}{\sigma_{m}^{2}}\mathbb{P}\{M=m\}\mathbb{E}\left[g_{m}(\mathbb{X})-g_{m-1}(\mathbb{X})\right]\\
 & =\frac{2}{\mathbb{E}[V^{2}]}\mathbb{E}\left[\sum_{m=1}^{\infty}\mathbb{P}\{N\geq m\}\left(g_{m}(\mathbb{X})-g_{m-1}(\mathbb{X})\right)\right]\\
 & =\frac{2}{\mathbb{E}[V^{2}]}\left(\mathbb{E}[g_{N}(\mathbb{X})]-g_{0}(\mathbb{X})\right)=\frac{2}{\mathbb{E}[V^{2}]}\left(\mathbb{E}[f(V)]-f(0)\right).
\end{align*}

Having shown that $V^{L}$ does in fact have the centered equilibrium distribution for $V$, we can apply Theorem \ref{1.2}
with $2B^{2}=\mathbb{E}[V^{2}]=\frac{\sigma^{2}}{\mu}$ to obtain
\begin{multline*}
d_{BL}(\mathscr{L}(V),\textnormal{{Laplace}}(0,B)) \leq\left(1+\frac{2}{B}\right)\mathbb{E}\left|V-V^{L}\right|\\
=\left(\mu^{-\frac{1}{2}}+\frac{\sqrt{8}}{\sigma}\right)\mathbb{E}\left|\left(X_{M}-X_{M}^{L}\right)+\textrm{sgn}(N-M)\sum_{i=(N\wedge M)+1}^{N\vee M}X_{i}\right|\\
\leq\left(\mu^{-\frac{1}{2}}+\frac{\sqrt{8}}{\sigma}\right)\left(\mathbb{E}\left|X_{M}-X_{M}^{L}\right|+\mathbb{E}\left|\sum_{i=(N\wedge M)+1}^{N\vee M}X_{i}\right|\right).
\end{multline*}

Finally, since the $X_{i}'s$ are independent with mean zero, the Cauchy-Schwarz inequality gives 
\begin{align*}
\mathbb{E}\left|\sum_{i=(N\wedge M)+1}^{N\vee M}X_{i}\right| & =\sum_{j=1}^{\infty}
\sum_{k=1}^{\infty}\mathbb{P}\{N\wedge M=j,\left|N-M\right|=k\}\mathbb{E}\left|\sum_{i=j+1}^{j+k}X_{i}\right|\\
 & \leq\sum_{j=1}^{\infty}\sum_{k=1}^{\infty}\mathbb{P}\{N\wedge M=j,\left|N-M\right|=k\}\mathbb{E}\left[\sum_{i=j+1}^{j+k}X_{i}^{2}\right]^{\frac{1}{2}}\\
 & \leq\sum_{j=1}^{\infty}\sum_{k=1}^{\infty}\mathbb{P}\{N\wedge M=j,\left|N-M\right|=k\}k^{\frac{1}{2}}\sup_{i\geq1}\sigma_{i}\\
 & =\sup_{i\geq1}\sigma_{i}\sum_{k=1}^{\infty}k^{\frac{1}{2}}\mathbb{P}\{\left|N-M\right|=k\}=\sup_{i\geq1}\sigma_{i}\mathbb{E}\left[\left|N-M\right|^{\frac{1}{2}}\right],
\end{align*}
and the proof is complete.
\end{proof}

\begin{rem*}
When the summands in Theorem \ref{4.3} have common variance $\sigma_{1}^{2}$, we have 
$\sigma^{2}=\mathbb{E}\left[\sum_{i=1}^{N}\sigma_{1}^{2}\right]=\mu\sigma_{1}^{2}$, so $M$ is defined by 
\[ \mathbb{P}\{M=m\}=\frac{\sigma_{m}^{2}}{\sigma^{2}}\mathbb{P}\{N\geq m\}=\frac{1}{\mu}\mathbb{P}\{N\geq m\}. \]
In other words, $M$ is the discrete equilibrium transformation of $N$ defined in \cite{PRR} for use in geometric approximation. 
Thus the $\mathbb{E}[\left|N-M\right|^{\frac{1}{2}}]$ term in the bound from Theorem \ref{4.3} may be regarded as measuring how close 
$\mathscr{L}(N)$ is to a geometric distribution. If the summands are i.i.d., then $X_{M}=_{d}X_{1}$ (as $M$ is assumed to be independent of all else), 
so Theorem \ref{1.2} shows that the $\mathbb{E}\left|X_{M}-X_{M}^{L}\right|$ term measures how close the distribution of the summands is to the Laplace. 
To put this into context, recall that if the $X_{i}'s$ are i.i.d. $\textrm{Laplace}(0,b)$ and $N\sim\textrm{Geometric}(p)$, then 
$p^{\frac{1}{2}}\sum_{i=1}^{N}X_{i}\sim\textrm{Laplace}(0,b)$.
\end{rem*}

We conclude our discussion with a proof of Theorem \ref{1.3}, which gives sufficient conditions for weak convergence in the setting of Theorem \ref{4.1}. 
Though it requires that the $X_{i}'s$ have uniformly bounded third absolute moments, the condition of symmetry is dropped and the identical distribution 
assumption is reduced to the requirement that the $X_{i}'s$ have common variance. This result is not quite as general as Theorem \ref{4.2}, 
but it does provide bounds on the error terms.

\begin{proof}[Proof of Theorem 1.3]
We are trying to bound the distance between the $\textrm{Laplace}(0,b)$ distribution and that of $p^{\frac{1}{2}}\sum_{i=1}^{N}X_{i}$ where $X_{1},X_{2},...$ are 
independent mean zero random variables with common variance $\mathbb{E}[X_{i}^{2}]=2b^{2}$ and uniformly bounded third absolute moments 
$\sup_{i\geq1}\mathbb{E}\left[\left|X_{i}\right|^{3}\right]=\rho<\infty$, and $N\sim\textrm{Geometric}(p)$ is independent of the $X_{i}'s$.

In the language of Theorem \ref{4.3}, we have $\mu=\frac{1}{p}$, $\sigma=b\sqrt{2\mu}$, and 
\[ \frac{\sigma_{m}^{2}}{\sigma^{2}}\mathbb{P}\{N\geq m\}=p\mathbb{P}\{N\geq m\}=p^{2}\sum_{i=m}^{\infty}(1-p)^{m-1}=p(1-p)^{m-1}=\mathbb{P}\{N=m\}, \]
so we can take $M=N$ to obtain 
\[ d_{BL}\left(\mathscr{L}\left(p^{\frac{1}{2}}\sum_{i=1}^{N}X_{i}\right),\textnormal{Laplace}(0,b)\right)
\leq\left(p^{\frac{1}{2}}+\frac{2p^{\frac{1}{2}}}{b}\right)\mathbb{E}\left|X_{N}-X_{N}^{L}\right|. \]

Applying Proposition \ref{3.4} gives
\begin{align*}
\mathbb{E}\left|X_{N}-X_{N}^{L}\right| & =\sum_{n=1}^{\infty}\mathbb{P}\{N=n\}\mathbb{E}\left|X_{n}-X_{n}^{L}\right|\\
 & \leq\sum_{n=1}^{\infty}\mathbb{P}\{N=n\}\left(\mathbb{E}\left|X_{n}\right|+\frac{1}{6b^{2}}\mathbb{E}\left[\left|X_{n}\right|^{3}\right]\right)\\
 & \leq\sum_{n=1}^{\infty}\mathbb{P}\{N=n\}\left(\mathbb{E}\left[X_{n}^{2}\right]^{\frac{1}{2}}+\frac{\rho}{6b^{2}}\right)=b\sqrt{2}+\frac{\rho}{6b^{2}}, 
\end{align*}
and the result follows
\end{proof}

\section*{Acknowledgements}

The authors would like to thank Larry Goldstein for suggesting the use of Stein's method to get convergence rates in the 
general setting of Theorem \ref{4.2} and for many helpful comments throughout the preparation of this paper. Thanks also to 
Alex Rozinov for several illuminating conversations concerning the material in Section 4 and to the anonymous referees whose 
careful notes were of great help to us. 

\bibliographystyle{plain}
\bibliography{SteLap}

\begin{thebibliography}{10}

\bibitem{BrPh}
Timothy~C. Brown and M.~J. Phillips.
\newblock Negative binomial approximation with {S}tein's method.
\newblock {\em Methodol. Comput. Appl. Probab.}, 1(4):407--421, 1999.
\newblock \href{http://www.ams.org/mathscinet-getitem?mr=1770372}{MR1770372}.

\bibitem{Chen}
Louis H.~Y. Chen.
\newblock Poisson approximation for dependent trials.
\newblock {\em Ann. Probability}, 3(3):534--545, 1975.
\newblock \href{http://www.ams.org/mathscinet-getitem?mr=0428387}{MR0428387}.

\bibitem{CGS}
Louis H.~Y. Chen, Larry Goldstein, and Qi-Man Shao.
\newblock {\em Normal approximation by {S}tein's method}.
\newblock Probability and its Applications (New York). Springer, Heidelberg,
  2011.
\newblock \href{http://www.ams.org/mathscinet-getitem?mr=2732624}{MR2732624}.

\bibitem{DiaHol}
Persi Diaconis and Susan Holmes, editors.
\newblock {\em {Stein's method: expository lectures and applications. Papers
  from the Workshop on Stein's Method held at Stanford University, Stanford,
  CA, 1998}}.
\newblock Institute of Mathematical Statistics Lecture Notes---Monograph
  Series, 46. Institute of Mathematical Statistics, Beachwood, OH, 2004.
\newblock \href{http://www.ams.org/mathscinet-getitem?mr=2118599}{MR2118599}.

\bibitem{Dob}
Christian D{\"o}bler.
\newblock Distributional transformations without orthogonality relations.
\newblock {\em ArXiv e-prints}, 2013.
\newblock \href{http://arxiv.org/abs/1312.6093}{arXiv:1312.6093}.

\bibitem{Ehm}
Werner Ehm.
\newblock Binomial approximation to the {P}oisson binomial distribution.
\newblock {\em Statist. Probab. Lett.}, 11(1):7--16, 1991.
\newblock \href{http://www.ams.org/mathscinet-getitem?mr=1093412}{MR1093412}.

\bibitem{GoRe1}
Larry Goldstein and Gesine Reinert.
\newblock Stein's method and the zero bias transformation with application to
  simple random sampling.
\newblock {\em Ann. Appl. Probab.}, 7(4):935--952, 1997.
\newblock \href{http://www.ams.org/mathscinet-getitem?mr=1484792}{MR1484792}.

\bibitem{GoRe2}
Larry Goldstein and Gesine Reinert.
\newblock Distributional transformations, orthogonal polynomials, and {S}tein
  characterizations.
\newblock {\em J. Theoret. Probab.}, 18(1):237--260, 2005.
\newblock \href{http://www.ams.org/mathscinet-getitem?mr=2132278}{MR2132278}.

\bibitem{Kal}
Vladimir Kalashnikov.
\newblock {\em {Geometric sums: bounds for rare events with applications. Risk
  analysis, reliability, queueing}}, volume 413 of {\em Mathematics and its
  Applications}.
\newblock Kluwer Academic Publishers Group, Dordrecht, 1997.
\newblock \href{http://www.ams.org/mathscinet-getitem?mr=1471479}{MR1471479}.

\bibitem{KKP}
Samuel Kotz, Tomasz~J. Kozubowski, and Krzysztof Podg{\'o}rski.
\newblock {\em {The {L}aplace distribution and generalizations. A revisit with
  applications to communications, economics, engineering, and finance}}.
\newblock Birkh\"auser Boston, Inc., Boston, MA, 2001.
\newblock \href{http://www.ams.org/mathscinet-getitem?mr=1935481}{MR1935481}.

\bibitem{Luk}
Ho~Ming Luk.
\newblock {\em {Stein's method for the {G}amma distribution and related
  statistical applications. Thesis (Ph.D.)--University of Southern
  California}}.
\newblock ProQuest LLC, Ann Arbor, MI, 1994.
\newblock \href{http://www.ams.org/mathscinet-getitem?mr=2693204}{MR2693204}.

\bibitem{PeRo}
Erol~A. Pek{\"o}z and Adrian R{\"o}llin.
\newblock New rates for exponential approximation and the theorems of {R}\'enyi
  and {Y}aglom.
\newblock {\em Ann. Probab.}, 39(2):587--608, 2011.
\newblock \href{http://www.ams.org/mathscinet-getitem?mr=2789507}{MR2789507}.

\bibitem{PRR}
Erol~A. Pek{\"o}z, Adrian R{\"o}llin, and Nathan Ross.
\newblock Total variation error bounds for geometric approximation.
\newblock {\em Bernoulli}, 19(2):610--632, 2013.
\newblock \href{http://www.ams.org/mathscinet-getitem?mr=3037166}{MR3037166}.

\bibitem{Ren}
Alfr{\'e}d R{\'e}nyi.
\newblock A characterization of {P}oisson processes.
\newblock {\em Magyar Tud. Akad. Mat. Kutat\'o Int. K\"ozl.}, 1:519--527
  (1957), 1957.
\newblock \href{http://www.ams.org/mathscinet-getitem?mr=0094861}{MR0094861}.

\bibitem{Ross}
Nathan Ross.
\newblock Fundamentals of {S}tein's method.
\newblock {\em Probab. Surv.}, 8:210--293, 2011.
\newblock \href{http://www.ams.org/mathscinet-getitem?mr=2861132}{MR2861132}.

\bibitem{Stn1}
Charles Stein.
\newblock A bound for the error in the normal approximation to the distribution
  of a sum of dependent random variables.
\newblock In {\em Proceedings of the {S}ixth {B}erkeley {S}ymposium on
  {M}athematical {S}tatistics and {P}robability ({U}niv. {C}alifornia,
  {B}erkeley, {C}alif., 1970/1971), {V}ol. {II}: {P}robability theory}, pages
  583--602. Univ. California Press, Berkeley, Calif., 1972.
\newblock \href{http://www.ams.org/mathscinet-getitem?mr=0402873}{MR0402873}.

\bibitem{Stn2}
Charles Stein.
\newblock {\em Approximate computation of expectations}.
\newblock Institute of Mathematical Statistics Lecture Notes---Monograph
  Series, 7. Institute of Mathematical Statistics, Hayward, CA, 1986.
\newblock \href{http://www.ams.org/mathscinet-getitem?mr=882007}{MR882007}.

\bibitem{Tod}
Alexis~A. Toda.
\newblock Weak limit of the geometric sum of independent but not identically
  distributed random variables.
\newblock {\em ArXiv e-prints}, 2011.
\newblock \href{http://arxiv.org/abs/1111.1786}{arXiv:1111.1786}.

\bibitem{vdvWel}
Aad~W. van~der Vaart and Jon~A. Wellner.
\newblock {\em {Weak convergence and empirical processes. With applications to
  statistics}}.
\newblock Springer Series in Statistics. Springer-Verlag, New York, 1996.
\newblock \href{http://www.ams.org/mathscinet-getitem?mr=1385671}{MR1385671}.

\end{thebibliography}

\end{document}